\documentclass[12pt,twoside,reqno]{amsart}
\usepackage{eurosym}
\usepackage[colorlinks=true,citecolor=blue]{hyperref}
\usepackage{mathptmx, amsmath, amssymb, amsfonts, amsthm, mathptmx, enumerate, color}
\usepackage{graphicx}
\usepackage{epstopdf}

\newtheorem{theorem}{Theorem}[section]
\newtheorem{corollary}{Corollary}[section]
\newtheorem{lemma}{Lemma}[section]
\newtheorem{proposition}{Proposition}[section]
\theoremstyle{definition}

\numberwithin{equation}{section}

\input{tcilatex}

\begin{document}
\title[Fenchel Subdifferential Operators]{Fenchel Subdifferential Operators:
A Characterization without Cyclic Monotonicity}
\author[J. E. Mart\'{\i}nez-Legaz]{ Juan Enrique Mart\'{\i}nez-Legaz$^{1}$}
\maketitle

\setcounter{page}{1}

\vspace*{1.0cm}

\vspace*{-0.6cm}

\begin{center}
{\footnotesize \textit{$^{1}$ Departament d'Economia i d'Hist\`{o}ria Econ%
\`{o}mica, Universitat Aut\`{o}noma de Barcelona, and BGS\ Math, Spain\\[0pt]
}}
\end{center}

\vskip4mm {\small \noindent \textbf{Abstract.} Fenchel subdifferential
operators of lower semicontinuous proper convex functions on real Banach
spaces are classically characterized as those operators that are maximally
cyclically monotone or, equivalently, maximally monotone and cyclically
monotone. This paper presents an alternative characterization, which does
not involve cyclic monotonicity. In the case of subdifferential operators of
sublinear functions, the new characterization substantially simplifies.
Dually, the new characterization of normal cone operators is very simple,
too.}


\renewcommand{\thefootnote}{} \footnotetext{%
E-mail addres: JuanEnrique.Martinez.Legaz@uab.cat.
\par
Received}

\section{Introduction}

\label{intro}

The Fenchel subdifferential is arguably the most fundamental notion in
convex analysis. The Fenchel subdifferential operator of a functional $%
f:X\rightarrow \mathbb{R\cup }\left\{ +\infty \right\} $ defined on a real
Banach space $X$ is%
\begin{eqnarray*}
&&%
\begin{array}{c}
\qquad \qquad \qquad \qquad%
\end{array}%
\begin{array}{c}
\partial f:X\rightrightarrows X^{\ast }%
\end{array}%
\begin{array}{c}
\qquad \qquad \qquad \qquad%
\end{array}
\\
&&%
\begin{array}{c}
\partial f\left( x\right) :=%
\end{array}%
\left\{ x^{\ast }\in X^{\ast }:f\left( y\right) \geq f\left( x\right)
+\left\langle y-x,x^{\ast }\right\rangle \text{ }\forall y\in X\right\} ;
\end{eqnarray*}%
here and in the sequel, $X^{\ast }$ is the dual space of $X$ and $%
\left\langle \cdot ,\cdot \right\rangle :X\times X^{\ast }\rightarrow 
\mathbb{R}$ denotes the duality product, that is, $\left\langle x,x^{\ast
}\right\rangle $ means the value of the continuous linear functional $%
x^{\ast }\in X^{\ast }$ at $x\in X.$ As is well known and easy to prove, $%
\partial f$ is cyclically monotone. Recall that a set-valued operator $%
T:X\rightrightarrows X^{\ast }$ is said to be cyclically monotone if%
\begin{eqnarray}
\dsum\limits_{i=0}^{k}\left\langle x_{i}-x_{i+1},x_{i}^{\ast }\right\rangle
&\geq &0\qquad \text{for every }\left( x_{i},x_{i}^{\ast }\right) \in T\quad
\left( i=0,1,...,k\right) ,  \label{cyclic} \\
&&%
\begin{array}{c}
\qquad%
\end{array}%
\begin{array}{c}
\text{with }k\geq 1\text{ arbitrary and }x_{k+1}:=x_{0};%
\end{array}
\notag
\end{eqnarray}%
here and throughout the whole paper, operators are identified with their
graphs, so that $\left( x,x^{\ast }\right) \in T$ means $x^{\ast }\in
T\left( x\right) .$ Every cyclically monotone operator is monotone, since
monotonicity corresponds to the case when $k=1$ in (\ref{cyclic}). Cyclic
monotonicity is closely connected to subdifferential operators, as
Rockafellar \cite[Theorem 1]{R66} proved that, given $T:X\rightrightarrows
X^{\ast },$ in order that there exist a proper convex functional $%
f:X\rightarrow \mathbb{R\cup }\left\{ +\infty \right\} $ such that $%
T\subseteq \partial f$, it is necessary and sufficient that $T$ be
cyclically monotone. Consequently, if $T$ is maximally (cyclically)
monotone, which means that $T$ is (cyclically) monotone and not properly
contained in any other (cyclically) monotone operator, then $T=\partial f.$
On the other hand, subdifferential operators of lower semicontinuous
(l.s.c., in brief) proper convex functionals are maximally monotone \cite[%
Theorem B]{R70a};\ therefore, one concludes that $T:X\rightrightarrows
X^{\ast }$ is the subdifferential operator of some l.s.c. proper convex
functional if and only if it is maximally (cyclically) monotone \cite[%
Theorem B]{R70a}. An extension of this characterization to suitably defined
subdifferentials of convex operators was obtained by Kusraev \cite{K78}.

The aim of this paper is to obtain an alternative characterization of
subdifferential operators not involving cyclic monotonicity. This is
achieved in Theorem \ref{main}. However, as one may expect, the new
characterization is not as simple and elegant as the one in \cite[Theorem B]%
{R70a}. It still involves maximal monotonicity, but the somewhat complicated
conditions i) - iii) of Proposition \ref{basic 3}, which replace cyclic
monotonicity, make the new characterization less attractive than the
classical one. By sharp contrast, in the case of subdifferential operators
of sublinear functionals, the new characterization, which does not involve
cyclic monotonicity either, is extremely simple and has a very easy proof.
Furthermore, since normal cones of closed convex sets are the
subdifferentials of their indicator functionals and the latter functionals
are the conjugates of the corresponding support functionals, which
characterize sublinear functionals, one easily obtains a simple
characterization of normal cone operators (Theorem \ref{normal cones}),
because subdifferentials of mutually conjugate functionals are inverse to
each other.

The rest of this paper is structured as follows. Section 2 contains
characterizations of normal cone operators and subdifferential operators of
l.s.c. proper sublinear functionals, and Section 3 characterizes
subdifferential operators of general l.s.c. proper convex functionals.

The notation and terminology used in the paper is mostly standard, but it is
explained here for the reader's convenience. The zero elements in $X$ and $%
X^{\ast }$ are denoted $0_{X}$ and $0_{X^{\ast }},$ repectively. The
projection of $X\times X^{\ast }$ onto $X^{\ast }$ is%
\begin{eqnarray*}
&&%
\begin{array}{c}
\Pi _{X^{\ast }}:X\times X^{\ast }\rightarrow X^{\ast }%
\end{array}
\\
&&%
\begin{array}{c}
\Pi _{X^{\ast }}\left( x,x^{\ast }\right) :=x^{\ast }.%
\end{array}%
\end{eqnarray*}%
The bidual space of $X$ is the dual $X^{\ast \ast }$ of $X^{\ast }.$ The
restriction of a functional $g:X^{\ast \ast }\rightarrow \mathbb{R\cup }%
\left\{ +\infty \right\} $ to $X$ (canonically identified with a subset of $%
X^{\ast \ast }$) is denoted $g_{|X}.$ The domain and the range of an
operator $T:X\rightrightarrows X^{\ast }$ are%
\begin{equation*}
dom~T:=\left\{ x\in X:T\left( x\right) \neq \emptyset \right\}
\end{equation*}%
and%
\begin{equation*}
range~T:=\dbigcup\limits_{x\in X}T\left( x\right) ,
\end{equation*}%
respectively. The inverse operator of $T$ is%
\begin{eqnarray*}
&&%
\begin{array}{c}
T^{-1}:X^{\ast }\rightrightarrows X%
\end{array}
\\
&&%
\begin{array}{c}
T^{-1}\left( x^{\ast }\right) :=\left\{ x\in X:x^{\ast }\in T\left( x\right)
\right\} .%
\end{array}%
\end{eqnarray*}%
The closure and the convex hull of a subset $C$ of a real Banach space $X$
are denoted $cl~C$ and $conv~C,$ repectively. Its barrier cone, its
recession cone and its indicator functional are%
\begin{equation*}
barr\left( C\right) :=\left\{ x^{\ast }\in X^{\ast }:\sup_{x\in
C}\left\langle x,x^{\ast }\right\rangle <+\infty \right\} ,
\end{equation*}%
\begin{equation*}
0^{+}\left( C\right) :=\left\{ d\in X:C+\mathbb{R}_{+}d=C\right\} ,
\end{equation*}%
and%
\begin{eqnarray*}
&&%
\begin{array}{c}
\delta _{C}:X\rightarrow \mathbb{R\cup }\left\{ +\infty \right\}%
\end{array}
\\
&&%
\begin{array}{c}
\delta _{C}\left( x\right) :=\left\{ 
\begin{array}{c}
0\text{\qquad if }x\in C, \\ 
+\infty \text{\qquad if }x\notin C,%
\end{array}%
\right.%
\end{array}%
\end{eqnarray*}%
respectively. The normal cone operator to $C$ is $N_{C}:=\partial \delta
_{C}.$ If $C\neq \emptyset ,$ its support functional is%
\begin{eqnarray*}
&&%
\begin{array}{c}
\sigma _{C}:X\rightarrow \mathbb{R\cup }\left\{ +\infty \right\}%
\end{array}
\\
&&%
\begin{array}{c}
\sigma _{C}\left( x^{\ast }\right) :=\sup_{x\in C}\left\langle x,x^{\ast
}\right\rangle .%
\end{array}%
\end{eqnarray*}%
In the case when $X$ is the dual of another real Banach space $Y,$ the
support functional $\sigma _{C}$ is defined on the bidual $Y^{\ast \ast },$
since $X^{\ast }=Y^{\ast \ast }$ in such a case. In the same way, in such a
situation $\partial \sigma _{C}$ is a set-valued operator from $X^{\ast }$
into $X^{\ast \ast }.$ The epigraph of a functional $f:X\rightarrow \mathbb{%
R\cup }\left\{ +\infty \right\} $ is the set%
\begin{equation*}
epi~f:=\left\{ \left( x,\alpha \right) \in X\times \mathbb{R}:f\left(
x\right) \leq \alpha \right\} .
\end{equation*}%
A functional $s:X\rightarrow \mathbb{R\cup }\left\{ +\infty \right\} $ is
said to be sublinear if it is convex and positively homogeneous, the latter
property meaning that for $x\in s^{-1}\left( \mathbb{R}\right) $ and $%
\lambda \geq 0$ one has $s\left( \lambda x\right) =\lambda s\left( x\right)
. $ Clearly, if $s$ is proper, then $s\left( 0\right) =0.$

The classical reference on convexity in finite dimension is Rockafellar's
book \cite{R70}. Convexity in Banach spaces has been the subject of many
excellent monographs, including \cite{BP78}, \cite{Z02} and the very recent 
\cite{MN22};\ the latter two books also consider functionals defined on real
locally convex topological vector spaces. Concerning monotonicity and its
close relationship with convexity, the interested reader may consult, for
instance, \cite{S08}, \cite{BI08} and, for operators defined on Hilbert
spaces, the more recent \cite{BC17}.

\vspace*{0.2cm}

\section{Normal Cone Operators of Closed Convex Sets and Subdifferentials of
Sublinear Functionals}

\label{sec::relaxed}

This section contains new and simple characterizations of normal cone
operators of closed convex sets and subdifferential operators of l.s.c.
sublinear functionals. The first result gives a simple sufficient condition
for a monotone operator to be contained in the normal cone operator of some
closed convex set.{\newline
}

\begin{proposition}
\label{basic}If $T:X\rightrightarrows X^{\ast }$ is monotone and $0_{X^{\ast
}}\in \dbigcap\limits_{x\in dom~T}T\left( x\right) ,$ then%
\begin{equation}
T\subseteq N_{cl~conv~dom~T}.  \label{incl}
\end{equation}
\end{proposition}

\begin{proof}
Let $\left( x,x^{\ast }\right) \in T$. For every $y\in dom~T,$ we have $%
0_{X^{\ast }}\in T\left( y\right) ;$ hence, by the monotonicity of $T,$ we
get $\left\langle y-x,x^{\ast }\right\rangle \leq 0.$ Thus,%
\begin{equation}
dom~T\subseteq \left\{ y\in X:\left\langle y-x,x^{\ast }\right\rangle \leq
0\right\} .  \label{incl 2}
\end{equation}%
Since the right hand side in (\ref{incl 2}) is a closed convex set, it
immediately follows that%
\begin{equation*}
cl~conv~dom~T\subseteq \left\{ y\in X:\left\langle y-x,x^{\ast
}\right\rangle \leq 0\right\} ,
\end{equation*}%
which, in view of $x\in dom~T\subseteq cl~conv~dom~T,$ implies that $x^{\ast
}\in N_{cl~conv~dom~T}\left( x\right) ,$ that is, $\left( x,x^{\ast }\right)
\in N_{cl~conv~dom~T}.$ This proves (\ref{incl}).
\end{proof}

\begin{corollary}
\label{basic 2}If $T:X\rightrightarrows X^{\ast }$ is monotone and $%
range~T=T\left( 0_{X}\right) $, then%
\begin{equation}
T\subseteq \partial \left( \sigma _{cl~conv~T\left( 0_{X}\right) }\right)
_{|X}.  \label{incl 3}
\end{equation}
\end{corollary}

\begin{proof}
The monotonicity of $T$ is equivalent to that of $T^{-1},$ and the
assumption $range~T=T\left( 0_{X}\right) $ is equivalent to the inclusion $%
0_{X}\in \dbigcap\limits_{x^{\ast }\in dom~T^{-1}}T^{-1}\left( x^{\ast
}\right) .$ Hence, using that%
\begin{equation*}
dom~T^{-1}=range~T=T\left( 0_{X}\right) ,
\end{equation*}%
Proposition \ref{basic}, applied to $T^{-1}$ (regarded as a set-valued
operator into $X^{\ast \ast }$), yields $T^{-1}\subseteq N_{cl~conv~T\left(
0_{X}\right) }=\partial \delta _{cl~conv~T\left( 0_{X}\right) }.$ Therefore,
for every $\left( x,x^{\ast }\right) \in T,$ we have%
\begin{equation*}
x\in T^{-1}\left( x^{\ast }\right) \subseteq \partial \delta
_{cl~conv~T\left( 0_{X}\right) }\left( x^{\ast }\right) .
\end{equation*}%
Thus, since $x\in X,$ we obtain%
\begin{equation*}
x^{\ast }\in \left( \partial \delta _{cl~conv~T\left( 0_{X}\right) }\right)
^{-1}\left( x\right) \subseteq \partial \left( \sigma _{cl~conv~T\left(
0_{X}\right) }\right) _{|X}\left( x\right) .
\end{equation*}%
This proves (\ref{incl 3}).
\end{proof}

\bigskip

Corollary \ref{basic 2} is to be compared to \cite[Theorem 1]{KK22}, which
establishes that a correspondence $T:\mathbb{R}^{n}\rightrightarrows \mathbb{%
R}^{n}$ (interpreted as assigning to each price vector $p$ a set of possible
production plans $T\left( p\right) $) is consistent with profit maximization
behavior, that is, there exists a convex closed production set $Y\subseteq 
\mathbb{R}^{n}$ such that for every price vector $p\in \mathbb{R}^{n}$ each
supply decision $z\in T\left( p\right) $ maximizes the scalar product $%
p\cdot y$ (i.e., the profit of producing $y$ under the given prices) subject
to $y\in Y$, if it satisfies the law of supply (i.e., it is monotone) and is
positively homogeneous of degree $0$ (i.e., $T\left( \lambda p\right)
=T\left( p\right) $ for every $p\in \mathbb{R}^{n}$ and $\lambda >0$).
Corollary \ref{basic 2} is simpler, as it does not require the homogeneity
condition;\ in its place, it has the assumption $range~T=T\left(
0_{X}\right) $, which would be an immediate consequence of positive
homogeneity of degree $0$ if imposing the mild extra hypothesis of $T^{-1}$
being closed-valued.

\begin{theorem}
\label{normal cones}Let $T:X\rightrightarrows X^{\ast }.$ There exists a
nonempty closed convex set\newline
$C\subseteq X$ such that $T=N_{C}$ if and only if $T$ is maximally monotone
and\newline
$0_{X^{\ast }}\in \dbigcap\limits_{x\in dom~T}T\left( x\right) .$
\end{theorem}

\begin{proof}
The "only if" statement is immediate. The "if statement" follows from
Proposition \ref{basic}, since $N_{cl~conv~dom~T}$ is monotone.
\end{proof}

\bigskip

The following theorem is related to \cite[Theorem 2]{KK22} in a similar way
as Corollary \ref{basic 2} is related to \cite[Theorem 1]{KK22},

\begin{theorem}
\label{sublinear}Let $T:X\rightrightarrows X^{\ast }.$ There exists an
l.s.c. proper sublinear functional\newline
$s:X\rightarrow \mathbb{R\cup }\left\{ +\infty \right\} $ such that $%
T=\partial s$ if and only if $T$ is maximally monotone and $range~T=T\left(
0_{X}\right) .$
\end{theorem}

\begin{proof}
The "only if" statement is immediate. The "if statement" follows from
Corollary \ref{basic 2}, since $\partial \left( \sigma _{cl~conv~T\left(
0_{X}\right) }\right) _{|X}$ is monotone.
\end{proof}

\vspace*{0.2cm}

\section{General Subdifferential Operators}

\label{sec::apriori}

To a given operator $A:X\times \mathbb{R\rightrightarrows }X^{\ast }\times 
\mathbb{R},$ we associate another operator $A_{X}:X\mathbb{\rightrightarrows 
}X^{\ast },$ defined by%
\begin{equation*}
A_{X}\left( x\right) :=\Pi _{X^{\ast }}\left( \left(
\dbigcup\limits_{\lambda \in \mathbb{R}}A\left( x,\lambda \right) \right)
\cap \left( X^{\ast }\times \left\{ -1\right\} \right) \right) .
\end{equation*}%
This section begins with two simple lemmas.

\begin{lemma}
\label{char epigraphs}Let $C\subseteq X\times \mathbb{R}$. There exists an
l.s.c. proper convex functional $f:X\rightarrow \mathbb{R\cup }\left\{
+\infty \right\} $ such that%
\begin{equation}
C=epi~f,  \label{epi = C}
\end{equation}%
if and only if the following conditions hold:

i) $C$ is nonempty, convex and closed,

ii)%
\begin{equation*}
\left( barr\left( C\right) \right) \cap \left( X^{\ast }\times \left\{
-1\right\} \right) \neq \emptyset ,
\end{equation*}

iii)%
\begin{equation*}
\left( 0_{X},1\right) \in 0^{+}\left( C\right) .
\end{equation*}
\end{lemma}

\begin{proof}
\textit{Only if. }Conditions i) and iii) are immediate. Condition ii)
follows from the fact that $\left( x^{\ast },-1\right) \in barr\left(
C\right) $ for every continuous affine minorant $\left\langle \cdot ,x^{\ast
}\right\rangle +b$ of $f.$

\textit{If. }Define $f:X\rightarrow \mathbb{R\cup }\left\{ +\infty ,-\infty
\right\} $ by%
\begin{equation*}
f\left( x\right) :=\inf \left\{ \lambda \in \mathbb{R}:\left( x,\lambda
\right) \in C\right\} .
\end{equation*}%
From i) and ii), it easily follows that $f$ is minorized by a continuous
affine functional , hence $f\left( x\right) >-\infty $ for every $x\in X.$
It is also clear that $C\subseteq epi~f,$ which, since $C\neq \emptyset ,$
implies that $f$ is proper. To see that the opposite inclusion also holds,
let $\left( x,\lambda \right) \in epi~f.$ Then, for every $\mu >\lambda ,$
there exists $\lambda ^{\prime }<\mu $ such that $\left( x,\lambda ^{\prime
}\right) \in C;$ hence, by iii), $\left( x,\mu \right) \in C.$ Letting $\mu
\rightarrow \lambda ^{+},$ we obtain that $\left( x,\lambda \right) \in C,$
since $C$ is closed according to i). We have thus proved (\ref{epi = C}).
\end{proof}

\begin{lemma}
\label{lambda = f(x)}Let $f:X\rightarrow \mathbb{R\cup }\left\{ +\infty
\right\} $ and $\left( x,\lambda \right) \in X\times \mathbb{R}.$ If $\left(
x^{\ast },-1\right) \in N_{epi~f}\left( x,\lambda \right) ,$ then $\lambda
=f\left( x\right) .$
\end{lemma}

\begin{proof}
If $\left( x^{\ast },-1\right) \in N_{epi~f}\left( x,\lambda \right) ,$ then 
$\left( x,\lambda \right) \in epi~f,$ since otherwise $N_{epi~f}\left(
x,\lambda \right) $ would be empty. It follows that $f\left( x\right) \leq
\lambda <+\infty ,$ and hence $\left( x,f\left( x\right) \right) \in epi~f.$
This inclusion, together with $\left( x^{\ast },-1\right) \in
N_{epi~f}\left( x,\lambda \right) ,$ yields%
\begin{equation*}
\left\langle \left( x,f\left( x\right) \right) -\left( x,\lambda \right)
,\left( x^{\ast },-1\right) \right\rangle \leq 0,
\end{equation*}%
which simply means that $\lambda \leq f\left( x\right) ,$ thus proving that $%
\lambda =f\left( x\right) .$
\end{proof}

\bigskip

The following corollary is an easy consequence of Lemma \ref{lambda = f(x)}.

\begin{corollary}
\label{nornal epi}If $f:X\rightarrow \mathbb{R\cup }\left\{ +\infty \right\} 
$ is convex and l.s.c., then $\left( N_{epi~f}\right) _{X}=\partial f.$
\end{corollary}

The next result gives sufficient conditions for the operator $A_{X}:X\mathbb{%
\rightrightarrows }X^{\ast }$ induced by a monotone operator $A:X\times 
\mathbb{R\rightrightarrows }X^{\ast }\times \mathbb{R}$ to be included in
the subdifferential operator of an l.s.c. convex functional .

\begin{proposition}
\label{basic 3}If $A:X\times \mathbb{R\rightrightarrows }X^{\ast }\times 
\mathbb{R}$ is monotone and satisfies:

i)%
\begin{equation*}
\left( barr\left( cl~conv~dom~A\right) \right) \cap \left( X^{\ast }\times
\left\{ -1\right\} \right) \neq \emptyset ,
\end{equation*}

ii)

\begin{equation*}
\left( 0_{X},1\right) \in 0^{+}\left( cl~conv~dom~A\right)
\end{equation*}%
and

iii)%
\begin{equation*}
\left( 0_{X^{\ast }},0\right) \in \dbigcap\limits_{\left( x,\lambda \right)
\in dom~A}A\left( x,\lambda \right) ,
\end{equation*}%
then the functional $f:X\rightarrow \mathbb{R\cup }\left\{ +\infty \right\} $
given by $epi~f=$ $cl~conv~dom~A$ is well defined and satisfies%
\begin{equation}
A_{X}\subseteq \partial f.  \label{incl 5}
\end{equation}
\end{proposition}

\begin{proof}
By Lemma \ref{char epigraphs}, $f$ is indeed well defined. Since $A$ is
monotone, by iii), Proposition \ref{basic} and (\ref{epi = C}), we have%
\begin{equation}
A\subseteq N_{epi~f}.  \label{incl 4}
\end{equation}

Let $\left( x,x^{\ast }\right) \in A_{X}.$ Then, by (\ref{incl 4}) and Lemma %
\ref{lambda = f(x)}, we have%
\begin{equation*}
\left( x^{\ast },-1\right) \in \dbigcup\limits_{\lambda \in \mathbb{R}%
}N_{epi~f}\left( x,\lambda \right) =N_{epi~f}\left( x,f\left( x\right)
\right) ,
\end{equation*}%
and hence $x^{\ast }\in \partial f\left( x\right) .$ This proves (\ref{incl
5}).
\end{proof}

\bigskip

The following result is the main one in this paper. It characterizes
subdifferential operators of general l.s.c. proper convex functionals within
the class of maximally monotone operators. Unlike the classical
characterization \cite[Theorem B]{R70a}, the new one does not involve cyclic
monotonicity.

\begin{theorem}
\label{main}Let $T:X\rightrightarrows X^{\ast }.$ There exists an l.s.c.
proper convex functional $f:X\rightarrow \mathbb{R\cup }\left\{ +\infty
\right\} $ such that $T=\partial f$ if and only if $T$ is maximally monotone
and there exists a monotone operator $A:X\times \mathbb{R\rightrightarrows }%
X^{\ast }\times \mathbb{R}$ satisfying conditions i) - iii) of Proposition %
\ref{basic 3} such that $T=A_{X}.$
\end{theorem}

\begin{proof}
To prove the "only if" statement, take $A:=N_{epi~f}.$ Since $%
dom~N_{epi~f}=epi~f$ and $epi~f$ is convex and closed, conditions i) and ii)
follow from Lemma \ref{char epigraphs}, whereas iii) is immediate. Moreover,
by Corollary \ref{nornal epi}, we have $\partial f=A_{X}.$

The "if statement" is an immediate consequence of Proposition \ref{basic 3},
since $\partial f$ is monotone.
\end{proof}

\vspace*{0.2cm}

\textbf{Acknowledgments}

I acknowledge financial support from the Spanish Ministry of Economy and
Competitiveness, through Grant PGC2018-097960-B-C21 and the Severo Ochoa
Programme for Centres of Excellence in R\&D (SEV-2015-0563). I am affiliated
with MOVE (Markets, Organizations and Votes in Economics). I am grateful to
Alex Kruger for his careful reading of this paper and many useful comments
(including his observation that Lemma \ref{lambda = f(x)} does not require
any assumption on $f$), which have helped me to improve the presentation.

\end{document}